\documentclass[10pt]{amsart}
\usepackage[a4paper, right=2.5cm, left=2.5cm, bottom=3cm, top=3cm]{geometry}
\usepackage[utf8]{inputenc}  
\usepackage{graphicx}
\usepackage{pgfplots}
\usepackage{amsmath,amssymb, mathtools}
\usepackage{multicol}
\usepackage{hyperref}
\usepackage{cleveref}
\usepackage{dsfont}
\usepackage{float}
\newtheorem{example}{Example}[section]

\newtheorem{theorem}[example]{Theorem}
\newtheorem{lemma}[example]{Lemma}
\newtheorem{corollary}[example]{Corollary}

\newtheorem{remark}[example]{Remark}
\newtheorem*{maintheorem*}{Main Theorem}
\allowdisplaybreaks
\numberwithin{equation}{section}

\newcommand{\N}{\mathbb{N}}

\newcommand{\R}{\mathbb{R}}

\newcommand{\pt}{\partial_t}
\newcommand{\px}{\partial_x}

\newcommand{\pxx}{\partial_{xx}^2}
\newcommand{\py}{\partial_y}

\newcommand{\ptx}{\partial_{tx}^2}

\newcommand{\weakstar}{\overset{\star}\rightharpoonup}

\DeclareMathOperator*{\esssup}{ess\sup}
\DeclareMathOperator*{\essinf}{ess\inf}

\newcommand{\loc}{\textnormal{loc}}

\newcommand{\eps}{\varepsilon}

\newcommand{\dd}{\,\mathrm{d}}
\renewcommand{\d}{\mathrm{d}}

\newcommand{\D}{{\mathcal D}}

\DeclareMathOperator{\BV}{BV}

\newcommand{\e}{\varepsilon}

\newcommand{\Lip}{\mathrm{Lip}}

\newcommand{\be}{\begin{equation}}
	\newcommand{\eq}{\end{equation}}

\newcommand{\weaks}{\stackrel{*}{\rightharpoonup}}

\newcommand{\ess}{\mathrm{ess}}
\newcommand{\Osc}{\mathrm{Osc}}

\begin{document}
                
\title[Nonlocal Ole\u{\i}nik estimates]{Ole\u{\i}nik-type estimates for nonlocal conservation laws \\ and applications to the nonlocal-to-local limit}

\keywords{Nonlocal conservation laws, nonlocal flux, singular limit, nonlocal-to-local limit, entropy condition, Ole\u{\i}nik condition. \vspace{1.5mm}}

\author[GMC]{Giuseppe Maria Coclite}
\address[G. M. Coclite]{Politecnico di Bari, Dipartimento di Meccanica, Matematica e Management, Via E. Orabona 4, 70125 Bari,  Italy.}
\email[]{giuseppemaria.coclite@poliba.it}

\author[MC]{Maria Colombo}
\address[M.~Colombo]{EPFL SB, Station 8, 1015 Lausanne, Switzerland.}
\email{maria.colombo@epfl.ch}

\author[GC]{Gianluca Crippa}
\address[G.~Crippa]{Universit\"at Basel, Departement Mathematik und Informatik,  Spiegelgasse 1, 4051 Basel, Switzerland.}
\email{gianluca.crippa@unibas.ch}

\author[NDN]{Nicola De Nitti}
\address[N. De Nitti]{Friedrich-Alexander-Universität Erlangen-Nürnberg, Department of Data Science, Chair for Dynamics, Control and Numerics (Alexander von Humboldt Professorship), Cauerstr. 11, 91058 Erlangen, Germany.}
\email{nicola.de.nitti@fau.de}

\author[AK]{Alexander Keimer}
\address[A. Keimer]{Friedrich-Alexander-Universität Erlangen-Nürnberg, Department of Mathematics, Cauerstr. 11, 91058 Erlangen, Germany.}
\email{alexander.keimer@fau.de}

\author[EM]{Elio Marconi}
\address[E.~Marconi]{Università degli Studi di Padova, Dipartimento di Matematica Tullio Levi-Civita, Via Trieste 63, 35121 Padova, Italy.}
\email{elio.marconi@unipd.it}

\author[LP]{Lukas Pflug}
\address[L. Pflug]{Friedrich-Alexander-Universität Erlangen-Nürnberg, (1) Competence Unit for Scientific Computing, Martensstr. 5a, 91058 Erlangen, Germany; (2) Department of Mathematics, Chair of Applied Mathematics (Continuous Optimization), Cauerstr. 11, 91058 Erlangen, Germany.}
\email{lukas.pflug@fau.de}

\author[LVS]{Laura V.~Spinolo}
\address[L.V. Spinolo]{IMATI-CNR, Via Ferrata 5, 27100 Pavia, Italy.}
\email{spinolo@imati.cnr.it}

\begin{abstract}
We consider a class of nonlocal conservation laws with exponential kernel and prove that quantities involving the nonlocal term  $W:=\mathds{1}_{(-\infty,0]}(\cdot)\exp(\cdot) \ast \rho$  satisfy an Ole\u{\i}nik-type  entropy condition. More precisely, under different sets of assumptions on the velocity function $V$, we prove that $W$ satisfies a one-sided Lipschitz condition and that $V'(W) W \partial_x W$ satisfies a one-sided bound, respectively. As a byproduct, we deduce that, as the exponential kernel is rescaled to converge to a Dirac delta distribution, the weak solution of the nonlocal problem converges to the unique entropy-admissible solution of the corresponding local conservation law, under the only assumption that the initial datum is essentially bounded and not necessarily of bounded variation.
\end{abstract}

\maketitle

\section{Introduction}
\label{sec:intro}
We study the nonlocal conservation law 
\begin{align}\label{eq:cl}
\begin{cases}
\partial_t \rho_\eps(t,x)  + \partial_x\big(V\big(W_{\eps}[\rho_\eps](t,x))\rho_\eps(t,x)\big)   	= 0,	& (t,x)\in  (0,T) \times \R, \\
\rho_\eps(0,x) = \rho_0(x),	&  x\in\R,
\end{cases}
\end{align}
with a velocity function $V:\R \to \R$ and an exponentially-weighted nonlocal impact  
\begin{align}\label{eq:W} 
    W_{\eps}[\rho_\eps](t,x)&\coloneqq\tfrac{1}{\eps}\int_{x}^{\infty} \exp \big(\tfrac{x-y}{\eps}\big)\rho_\eps(t,y)\dd y , &(t,x)\in (0,T)\times\R,
\end{align}
where $\eps >0$ and $T>0$. We note that, for \((t,x)\in(0,T)\times\R\), the nonlocal term $W_{\eps}$ satisfies the following equation:
\begin{align}\label{eq:WWxq}
    \partial_{x}W_{\eps}[\rho_{\eps}](t,x)&=\partial_{x} \tfrac{1}{\eps}\int_{x}^{\infty}\exp\big(\tfrac{x-y}{\eps}\big)\rho_{\eps}(t,y)\dd y=\tfrac{1}{\eps}W_{\eps}[\rho_{\eps}](t,x)-\tfrac{1}{\eps}\rho_{\eps}(t,x).
\end{align}

The existence and uniqueness of solutions for nonlocal conservation laws have been thoroughly analyzed in recent years: we refer to \cite{MR3447130,MR4502637,MR3670045,keimer2018existence,MR3890783} and references therein for an overview. Furthermore, the convergence of nonlocal conservation laws to the corresponding local models as the nonlocal weight tends to a Dirac delta distribution has attracted much attention. Several results in this direction are available in the literature (see~\cite{MR4110434,MR4283539,2012.13203,MR4265719,MR4340167,MR4300935,2206.03949, MR3961295, MR3944408}). In particular, the most recent ones -- \cite{2012.13203,2206.03949} -- provide satisfactory answers in case the initial datum has bounded total variation. 

Our main aim is to prove \emph{Ole\u{\i}nik-type inequalities} for quantities involving the nonlocal term $W_\eps$. Then, we use them to prove that, as $\eps \to 0^+$, the solution of \eqref{eq:cl} converges to the unique entropy admissible solution of the (local) conservation law 
\begin{align}\label{eq:cll}
\begin{cases}
\partial_t \rho(t,x)  + \partial_x\big( V(\rho(t,x))\rho(t,x)\big)   	= 0,	& (t,x)\in  (0,T) \times \R, \\ 
\rho(0,x) = \rho_0(x),	&  x\in\R,
\end{cases}
\end{align}
assuming that the initial datum is not necessarily of bounded variation, but only essentially bounded, which is a novel contribution compared to the previous literature. A main point in the study of this singular limit problem is establishing the precompactness in $L^1_{\loc}$ of the solutions $\rho_\e$ of the nonlocal equation.  In our approach, this is a consequence of the maximum principle (uniform in $\eps$) and of the Ole\u{\i}nik-type estimate, which also rules out the emergence of non-entropic shocks, thus leading to the entropy admissibility of the accumulation points of the family $\rho_\eps$ as~$\eps \to 0^+$.

For the scalar (local) conservation law 
\begin{align}\label{eq:scl} 
\partial_t \rho(t,x)+\partial_x f(\rho(t,x))=0, \qquad (t,x) \in (0,T)\times\R,
\end{align}
the celebrated result by Ole\u{\i}nik \cite{MR0151737} (see also the following contributions, which are contemporary to Ole\u{\i}nik's work: Lax, \cite{MR93653}; Lady\v{z}enskaya \cite{MR0094539}; and Hopf \cite{MR47234}) states that if \(f\) is uniformly strictly convex, i.e. \(f^{\prime \prime}(\cdot) \geq \kappa>0\) on \(\mathbb{R}\), then any entropy admissible solution of \eqref{eq:scl} satisfies the following one-sided Lipschitz estimate:
\begin{align*}
\rho(t, y)-\rho(t, x) \leq \frac{y-x}{\kappa t}, \qquad  t>0, \   x,y\in\R,\  x \leq y.
\end{align*}
The Ole\u{\i}nik estimate provides an equivalent characterization of entropy solutions and is an  example of the fact that the nonlinearity of the PDE provides a regularizing
effect on the solution: indeed, as this upper estimate only allows for decreasing jumps, it implies that \(L^{\infty}\) data are instantaneously regularized to functions of locally bounded variation \(\left(\mathrm{BV}_{\mathrm{loc}}\right)\). On the contrary, a linear flux $f(\rho) = b\rho$ (with $b \in \R$) does not generate
additional regularity as  the solution is simply a translation of the initial datum: $\rho(t, x) = \rho(0,x - bt)$.

This inequality can be written in a `sharp' form (see \cite{MR0481581,MR688972}): when \(f^{\prime \prime}(\cdot) \geq 0\) and moreover there are no non-trivial intervals where $f$ is affine (\emph{Tartar's condition} \cite{MR584398}), we have
\begin{align*}
f^{\prime}(\rho(t, y))-f^{\prime}(\rho(t, x)) \leq \frac{y-x}{t}, \qquad  t>0, \  x,y\in\R,\  x \leq y .
\end{align*}
Several inequalities of Ole\u{\i}nik type have been established for non-convex (or non-concave) fluxes as well as for some systems of conservation laws (see, e.g., \cite{MR1632980,MR818862, MR2405854,  MR1855004,MR2119939}). 

As Lax observed in~\cite{MR66040}, the Ole\u{\i}nik inequality implies  the compactness in $L^1_{\loc}$ of the semigroup $(S_t)_{t >0}$ of entropy weak solutions to strictly convex scalar conservation laws in one space dimension. More recently, quantitative estimates of the compactness of $S_t$ have been established by relying on the notion of Kolmogorov $\eps$-entropy (see~\cite{MR2954617,MR2142881}).

For nonlocal conservation laws, inequalities of the type listed above are not known to date. In this direction, the only result available in the literature is \cite[Theorem 3]{MR4300935}, where an Ole\u{\i}nik-type estimate is obtained under the strong assumptions that the initial datum itself satisfies a one-sided Lipschitz condition and is bounded away from zero; and \cite[Theorem 3.10]{MR4172728} (for a slightly different, but related, class of nonlocal equations, namely nonlocal transport equations), under the rather restrictive assumptions that the initial datum is quasi-concave and has an upper bound on the derivative.

\subsection{Outline}
\label{ssec:outline}
The paper is organized as follows. In Section \ref{sec:main}, we present the statements of our main results, namely, the Ole\u{\i}nik-type inequalities involving $W_\eps$ and $V'(W_\eps) W_\eps \partial_x W_\eps $. 

The proof of these inequalities is contained in Section \ref{sec:proof}. As a byproduct, in Section \ref{sec:proof-limit}, we prove the nonlocal-to-local convergence for initial data in $L^\infty$.  
Finally, in Section~\ref{sec:numerics}, we conclude by presenting some numerical experiments.

\section{Main results}
\label{sec:main}

Our main results are the following Ole\u{\i}nik-type estimates involving the  nonlocal term $W_\eps$. More precisely, under different sets of assumptions on the velocity function $V$, we show that $W_\eps$ satisfies a one-sided Lipschitz condition and that $V'(W_\eps) W_\eps \partial_x W_\eps $ satisfies a one-sided bound, respectively. 

\begin{theorem}[Ole\u{\i}nik-type inequality for $W_\eps$]\label{th:o}
Let $0<\kappa_1<\kappa_2$ and $ \rho_{0}\in L^{\infty}(\R;\R_{\geq0})$ and let $V\in W^{2,\infty}_{\loc}(\R)$ be a nonincreasing velocity function such that at least one of the following conditions is satisfied: 
\begin{align}
\label{ass:linear}&V'(\xi) = - \delta < 0, && \forall \xi \in [    \essinf \rho_{0},    \esssup \rho_{0}]; \\
\label{ass:v-conv-more}&0\le V'(\xi)+V''(\xi)\xi \le \kappa_1, \quad V'(\xi) \le - \kappa_2,  \quad  \kappa_2 - \kappa_1>0, && \forall \xi \in [    \essinf \rho_{0},    \esssup \rho_{0}].
\end{align} 
Let  $\rho_{\e}$ be the solution of the Cauchy problem associated to \eqref{eq:cl}.
Then the nonlocal term $W_\eps$ satisfies the following inequality:
\begin{align}\label{eq:ol}
\frac{W_\eps(t,x)-W_\eps(t,y)}{x-y} \ge - \frac{1}{\kappa t}, \qquad \text{for all $t >0$ and $x,y \in \R$ with $x \neq y$,}
\end{align}
with $\kappa \coloneqq \delta$ (in case assumption \eqref{ass:linear}holds) or $\kappa \coloneqq \kappa_2 - \kappa_1$ (in case assumption \eqref{ass:v-conv-more} holds).
\end{theorem}

\begin{remark}[Convexity/concavity assumptions] \label{rk:convex}
If we assume that the flux is strictly convex (instead of strictly concave as implied by assumptions \eqref{ass:linear} or \eqref{ass:v-conv-more}), the velocity increasing, and the convolution looking to the left, we can establish analogous results. In particular, for the case of a convex flux with linear velocity  (i.e., the counterpart of the setting of \eqref{ass:linear}), we refer to~\cite{nwave2023}.

Here, we consider the concave case because of its relevance for traffic models (see \cite{MR2328174}). 

 \end{remark}

\begin{theorem}[Ole\u{\i}nik-type inequality for $V'(W_\eps) W_\eps \partial_x W_\eps$] \label{th:og}
Let $0<\kappa_1$ and $ \rho_{0}\in L^{\infty}(\R;\R_{\geq0})$ and let $V\in W^{2,\infty}_{\loc}(\R)$ be a nonincreasing velocity function such that at least one of the following conditions is satisfied:  
\begin{align}
\label{ass:ob2} &0\le (-V'(\xi)-V''(\xi)\xi)
( \esssup \rho_{0} -     \essinf \rho_{0}) \le -V'(\xi)\xi, && \forall \xi \in [    \essinf \rho_{0},    \esssup \rho_{0}]; \\
\label{ass:ob3}
&-V'(\xi) \le V''(\xi)\xi \le -(2- \kappa_1)V'(\xi), && \forall \xi \in [    \essinf \rho_{0},    \esssup \rho_{0}]   . 
		\end{align}
Let  $\rho_{\e}$ be the solution of the Cauchy problem associated to \eqref{eq:cl}. 
Then,
	\be \label{e:claim}
	\sup_{\R}  V'(W_\eps) W_\eps \partial_x W_\eps \le \frac{\|\rho_{0}\|_{L^\infty(\R)}}{\kappa t}, \quad \text{ for all $t>0$,}
	\eq
	where 
 $\kappa:= 1$ (in case  assumption \eqref{ass:ob2} holds) or $\kappa:= \kappa_1$ (in case assumption~\eqref{ass:ob3} holds).
\end{theorem}

\begin{remark}[Independence of the constant on $\mathrm{TV}(\rho_0)$]
In Theorems \ref{th:o} and \ref{th:og}, the initial datum is not required to be of bounded variation. 
\end{remark}

\begin{remark}[Assumptions on the velocity function and traffic models]\label{rk:traffic}
The assumptions on the velocity function $V$  in Theorems~\ref{th:o} and \ref{th:og} may look quite restrictive. In the proofs, we exploit such conditions when manipulating the equations satisfied by $\px W_\eps$ and $V'(W_\eps) W_\eps \partial_x W_\eps$ to deduce a Riccati-type differential inequality. Despite their apparent intricacy, these assumptions are satisfied by several classes of well-known traffic models, possibly under some restrictions on the initial data.

\begin{enumerate}
       \item 	Assumption \eqref{ass:linear} is satisfied by  the  \emph{Greenshield model}, $V(\xi) = v_{\mathrm{max}}(1-\xi/\rho_{\max})$
     (see \cite[Chapter~3, Eq. (3.1.3)]{MR2328174}).

\item 	The \emph{Underwood model} $V(\xi) = v_0 e^{\left( -\frac{\xi}{\rho_{\max}} \right)}$, with $\rho_{\max}>0$ and $v_0 >0$ (see \cite[Chapter 3, Eq.~(3.1.5)]{MR2328174}), satisfies Assumption \eqref{ass:ob2} under the constraint $     \essinf \rho_{0} \ge \frac{3-\sqrt 8}{2}     \esssup \rho_{0}$.

\item The generalized \emph{Greenshield model} $V(\xi)= v_0 \left(1 - \left( \frac{\xi}{\rho_{\max}}\right)^n \right)$, with $\rho_{\max}>0$ and $v_0 >0$  (see \cite[Chapter 3, Eq.~(3.1.6)]{MR2328174}), satisfies Assumption \eqref{ass:ob2} under the constraint $     \essinf \rho_{0} \ge \frac{n}{n+1}     \esssup \rho_{0}$. 
 
\item The \emph{generalized California model} $V_\alpha(\xi) = v_0 \left( \frac{1}{\xi^\alpha} - \frac{1}{\rho_{\max}^\alpha}\right)$, with $\rho_{\max}>0$ and $v_0 >0$ and $\alpha \in (0,1)$ (cf. \cite[Chapter 3, Eq. (3.1.7)]{MR2328174}), satisfies Assumptions \eqref{ass:v-conv-more} and  \eqref{ass:ob3}. This velocity is not locally Lipschitz continuous at $\xi = 0$;
however, its variant  $V_\alpha(\xi) = v_{\max} \left( \frac{1}{\xi^\alpha + \frac{v^\alpha_{\max}}{v^\alpha_{\max} +1}} - \frac{1}{\rho_{\max}^\alpha}\right)$ is and satisfies the same assumption; alternatively, we may just assume $\rho_0 \ge c_0 >0$.
\end{enumerate}
\end{remark}

As a consequence of Theorems~\ref{th:o} and~\ref{th:og}, we deduce the following nonlocal-to-local convergence results. The key difference compared to \cite{2012.13203,2206.03949} is the fact that we do not require the initial datum to have bounded total variation; on the other hand, some extra assumptions on the velocity function are required.

\begin{corollary}[Nonlocal-to-local singular limit problem]\label{T_main} Let us suppose that either 
\begin{itemize}
    \item[--] the assumptions of Theorem \ref{th:o} hold; 
    \item[--]  the assumptions of Theorem \ref{th:og} hold, and additionally $V' \le -\kappa_2 <0$ for some $\kappa_2 >0$. 
\end{itemize}
Let $\rho_{\e}$ be the unique weak solution of the nonlocal conservation law \eqref{eq:cl} and $\rho$ be the unique entropy admissible solution of the local conservation law~\eqref{eq:cll}. Then, both $\rho_{\e}$ and the corresponding nonlocal term $W_\eps$ converge to $\rho$ in $L^1_{\loc}([0,T)\times \R)$. 
\end{corollary}

Before diving into the proof of our main results, let us recall the following well-posedness result and some fundamental properties of the nonlocal conservation law \eqref{eq:cl}. In particular, we remark that the nonlocal term $W_\eps$ has additional regularity and satisfies a local transport equation with nonlocal source.  We refer to \cite[Theorem 2.1 \& Lemma 3.1]{2012.13203} (which, in turn, relies in part on \cite[Theorem 2.20 \& Theorem~3.2 \& Corollary 4.3]{MR3670045} or \cite[Theorem 2.1 \& Corollary 2.1]{MR4502637}), \cite[Theorem 2.1]{MR3461737},  \cite[Proposition 2.1 \& Corollary 2.2]{2206.03949}, or \cite{nwave2023} for the proof of a similar statement.

\begin{theorem}[Existence and uniqueness of weak solutions, maximum principle, and properties of the nonlocal term]\label{th:wp}
Let $ \rho_{0}\in L^{\infty}(\R;\R_{\geq0}) $ and let $V\in W^{2,\infty}_{\loc}(\R)$ be a non-increasing velocity function. Then, for every~$\eps>0$, there is a unique weak solution \(\rho_\eps\in C\big([0,T];L^{1}_{\loc}(\R)\big)\cap L^{\infty}((0,T);L^{\infty}(\R))\) of the nonlocal conservation law \eqref{eq:cl}. Also, the maximum principle holds:
\begin{align}\label{eq:uniform_bound}
    \essinf_{x\in\R} \rho_{0}(x)\leq \rho_\eps(t,x)\leq \esssup_{x \in \R} \rho_0(x), \quad  \text{ for a.e. } (t,x)\in (0,T)\times\R .
\end{align}

Moreover, the nonlocal term $W_{\eps}$ satisfies the following properties:
\begin{enumerate}
    \item $W_{\eps} \in W^{1,\infty}\left([0,T] \times \mathbb{R}\right)$ and
$\essinf \rho_{0} \le W_\eps \le \esssup \rho_0$;
    \item $W_{\eps} \in C^{0}\left([0,T]; L^{1}_{\mathrm{loc}}(\mathbb{R})\right)$;
    \item if $\rho_{0} \in C^k(\mathbb{R})$, then $W_{\eps} \in C^{k+1}\left([0,T] \times \mathbb{R}\right)$ for $k \ge 0$.
\end{enumerate} 
In addition, for every $t \in [0, T]$,  the map $t\mapsto \Lip^-(\rho_{\eps}(t,\cdot))$  is a locally Lipschitz continuous function from $[0, + \infty)$ to $[0,+\infty)$.  Here, $\mathrm{Lip}^-(\rho_{\eps}) := - \inf\limits_{x<y} \frac{\rho_{\eps}(y) - \rho_{\eps}(x)}{y-x}.$
Furthermore, $W_\eps$  satisfies the following transport equation almost everywhere:  
\begin{align}\label{eq:Wt}
\begin{cases}
\partial_t W_\eps(t,x)+V(W_\eps(t,x))\partial_x W_\eps(t,x) & {}\\
   \qquad =-\frac{1}{\eps}\int_{x}^{\infty}\exp(\frac{x-y}{\eps})V'(W_\eps(t,y))\partial_y W_\eps(t,y)W_\eps(t,y)\dd y,  & (t,x)\in(0,T)\times\R,\\
   W_{\eps}(0,x)=\frac{1}{\eps}\int_{x}^{\infty}\exp(\frac{x-y}{\eps})\rho_{0}(y)\dd y, & x\in\R.
   \end{cases}
\end{align}
We remark that \eqref{eq:Wt} can be equivalently rewritten as 
\begin{equation}\label{e:nle*}
\partial_t W_\eps + \partial_x(V(W_\eps) W_\eps )= g_\eps- g_\eps\ast \eta_\eps, \qquad \mbox{provided } \ g_\eps= V'(W_\eps) W_\eps \partial_x W_\eps,  
\end{equation}
and we use the notation
\begin{equation}
    \label{e:defetaeps}
\eta(\cdot) \coloneqq \mathds{1}_{(-\infty,0]}(\cdot)\exp(\cdot), \quad  \eta_\eps \coloneqq \eps^{-1}\eta(\cdot/\eps).
\end{equation}
\end{theorem}

\section{Proof of the Ole\u{\i}nik estimates}
\label{sec:proof}

In order to prove the Ole\u{\i}nik estimates, it is helpful to regularize the initial data of the nonlocal conservation law \eqref{eq:cl}. To this end, we need the following stability result (see \cite[Theorem 3.1]{2012.13203} and \cite{nwave2023} for related results).

\begin{lemma}[Approximation]\label{theo:stability}
Let us consider the Cauchy problem 
\begin{align}\label{eq:cp}
\begin{cases}
\partial_{t} \rho(t,x)+\partial_{x}(V(W[\rho](t,x)) \rho(t,x))=0, & (t,x) \in (0,+\infty) \times \R,  \\
\rho(0, x)=\rho_{0}(x), & x \in \R,
\end{cases}
\end{align}
where $$W[\rho](t, x)\coloneqq\int_{x}^{+\infty} \exp(x-y) \rho(t, y) \dd y, \qquad (t,x)\in(0,\infty)\times\R.$$ 
Let us also consider the family of the Cauchy problems
\begin{align}\label{eq:cpn}
\begin{cases}
\partial_{t} \rho_{n}(t,x)+\partial_{x}\left(V\left(W_{n}(t,x)\right) \rho_{n}(t,x)\right)=0, & (t,x) \in (0,+\infty) \times \R, \\
\rho_{n}(0, x)=\rho_{0, n}(x), & x \in\R,
\end{cases}
\end{align}
where  $n \in \N$ and  
$$W_{n}[\rho_n](t, x):=\int_{x}^{+\infty} \exp(x-y) \rho_{n}(t, y) \dd y.
$$
Let us furthermore assume that, for a suitable constant $M>0$, it holds  
\begin{align}\label{assl:linfty}
0 \leq \rho_{0, n} \leq M \text { a.e.  for every $n$, } \qquad \rho_{0, n} \weakstar \rho_{0} \text { weakly-*  in } L^{\infty}(\mathbb{R}) \text{ for } n\rightarrow\infty .
\end{align}
Then, 
$$
W_{n} \rightarrow W \quad \text { strongly in } L_{\mathrm{loc }}^{1}\left(\mathbb{R}_{+} \times \mathbb{R}\right).
$$
\end{lemma}
\begin{remark}[More general kernels]\label{rk:stability-eta}
The statement of Lemma \ref{theo:stability} is still valid if we replace the exponential weight with a more general kernel
$$
\eta \in \operatorname{Lip}\left(\mathbb{R}_{-}\right), \quad \int_{\mathbb{R}_{-}} \eta(y) \dd y=1, \quad \eta^{\prime} \geq 0.
$$
\end{remark}

\begin{proof}[Proof of Lemma~\ref{theo:stability}] By the maximum principle, the first condition in \eqref{assl:linfty} yields
\begin{align}\label{eq:04}
0 \leq \rho_{n}, W_{n} \leq M \text { a.e. and for every  $n$.}
\end{align}
Owing to \eqref{eq:04}, we have that, up to subsequences, $\rho_{n} \weakstar v$ in the weak-* topology of $L^{\infty}\left(\mathbb{R}_{+} \times \mathbb{R}\right)$, for some bounded limit function $v$. By Lebesgue's Dominated Convergence Theorem, this, in turn, implies that $W_{n} \rightarrow v \ast \mathds{1}_{(-\infty,0]}(\cdot)\exp(\cdot)$ strongly in $L_{\mathrm{loc}}^{1}\left(\mathbb{R}_{+} \times \mathbb{R}\right)$. By passing to the limit in the distributional formulation of \eqref{eq:cpn}, we conclude that $v$ coincides with the unique bounded distributional solution of \eqref{eq:cp}. This concludes the proof of the lemma.
\end{proof}

\begin{remark}[Continuity in time]\label{rk:time-c}
By using \cite[Lemma 1.3.3]{MR3468916}, we can assume -- with no loss of generality -- that the functions $t \mapsto \rho(t, \cdot)$ and $t \mapsto W(t, \cdot)$ are continuous from $\mathbb{R}_{+}$ to $L^{\infty}(\mathbb{R})$ endowed with the $L^{\infty}$-weak-* and the strong $L_{\mathrm{loc}}^{1}$ topology, respectively. In Section \ref{sec:proof-limit}, we will use this remark to pass to the limit in the nonlocal Ole\u{\i}nik inequalities \eqref{eq:ol} or \eqref{e:claim} for every $t>0$.
\end{remark}

\subsection{Ole\u{\i}nik-type estimate for \texorpdfstring{$W_\eps$}{W}}
\label{ssec:proof-1}

In this section, we prove  Theorem \ref{th:o}. The basic idea is to use the transport equation with nonlocal source satisfied by $W_\eps$, i.e.~\eqref{eq:Wt}.

\begin{proof}[Proof of Theorem \ref{th:o}]
Owing to Lemma \ref{theo:stability}, it suffices to prove the statement for initial data $\rho_{0} \in \mathcal D \cap C^2(\R)$
and thus for solutions $\rho_{\e} \in C^2([0,T]\times \R)$. Here, 
\begin{equation}\label{e:setD}
\mathcal D\coloneqq \big\{ \rho_{0} \in L^\infty(\R): \mathrm{TV}(\rho_{0})<\infty, \, \rho_{0}(x)\in [0, \rho_{\max}] \, \mbox{ for a.e. }x\in \R \big\}.
\end{equation}
By differentiating \eqref{eq:Wt} with respect to $x$ we get
\begin{align}\label{eq:Wx}
\begin{split}
    \ptx W_\eps = &- V(W_\eps)\pxx W_\eps - V'(W_\eps)(\px W_\eps)^2 + \frac{1}{\eps} V'(W_\eps) W_\eps \px W_\eps \\ & - \tfrac{1}{\eps^2} \int_x^\infty \exp\left(\tfrac{x-y}{\eps}\right)V'(W_\eps)W_\eps \py W_\eps \dd y. 
\end{split}
\end{align}
We now set $m(t) := \min_{y \in \R} \partial_y W_{\eps}(t,y)$ and assume  without loss of generality that $m(t) \le 0$. \\
\textbf{Case 1:} \emph{we assume \eqref{ass:v-conv-more}.} We estimate the right-hand side of \eqref{eq:Wx} from below as follows: 
\begin{align*}
    \ptx W_\eps = &- V(W_\eps)\pxx W_\eps - V'(W_\eps)(\px W_\eps)^2 + \tfrac{1}{\eps} V'(W_\eps) W_\eps \px W_\eps \\ & - \tfrac{1}{\eps^2} \int_x^\infty \exp\left(\tfrac{x-y}{\eps}\right)V'(W_\eps)W_\eps \py W_\eps \dd y\\
    \geq& - V(W_\eps)\pxx W_\eps - V'(W_\eps)(\px W_\eps)^2 + \tfrac{1}{\eps} V'(W_\eps) W_\eps \px W_\eps \\ 
    &\ - \tfrac{1}{\eps^2} m\int_x^\infty \exp\left(\tfrac{x-y}{\eps}\right)V'(W_\eps)W_\eps \dd y
    \intertext{(integrating by parts in the last term)}
    =&- V(W_\eps)\pxx W_\eps - V'(W_\eps)(\px W_\eps)^2 + \tfrac{1}{\eps} V'(W_\eps) W_\eps \px W_\eps \\ 
    &\ -\tfrac{1}{\eps}mV'(W_{\eps})W_{\eps} -\tfrac{1}{\eps}m\int_{x}^{\infty}\exp\left(\tfrac{x-y}{\eps}\right)\big(V'(W_{\eps})\partial_{y}W_{\eps}+V''(W_{\eps})W_{\eps}\partial_yW_{\eps}\big)\dd y.
\end{align*}
Let us consider $\bar x \in \R$ such that $m(t) = \px W_{\eps}(t,\bar x)$ (we then know that \(\partial_{xx}^{2}W_{\eps}(t, \bar x)=0\)) and evaluate the previous expression at $x = \bar x$.
Due to \eqref{ass:v-conv-more}, we have 
\[
-\tfrac{1}{\eps}m\int_{x}^{\infty}\exp\left(\tfrac{x-y}{\eps}\right)\big(V'(W_{\eps})+V''(W_{\eps})W_{\eps}\big)\partial_{y}W_{\eps}\dd y 
\ge -\kappa_1 m^2
\]
and, then, we deduce  
\begin{align*}
    \tfrac{\d}{\d t} m(t) &\ge -V'(W_{\eps})m(t)^{2}-\kappa_1 m^2(t) \ge (\kappa_2 - \kappa_1) m^2(t), \qquad t >0.
\end{align*}
\textbf{Case 2:} \emph{we assume \eqref{ass:linear}.} We estimate the right-hand side of \eqref{eq:Wx} from below as follows:  
\begin{align*}
    \ptx W_\eps & = -V(W_\eps) \pxx W_\eps +\delta (\px W_\eps)^2 - \tfrac{\delta}{\eps}  W_\eps \px W_\eps \\ & \quad + \tfrac{\delta}{\eps^2} \int_x^\infty \exp\left(\tfrac{x-y}{\eps}\right)W_\eps \py W_\eps \dd y\\
    &= - V(W_\eps)\pxx W_\eps +\delta (\px W_\eps)^2 -  \tfrac{\delta}{\eps}  W_\eps \px W_\eps \\ 
    &\quad +  \tfrac{\delta}{\eps^2} \int_x^\infty \exp\left(\tfrac{x-y}{\eps}\right)\Big(\eps \py W^\eps(t,y) + \rho_\eps(t,y) \Big) \py W_\eps \dd y\\
    &= - V(W_\eps)\pxx W_\eps +\delta (\px W_\eps)^2 -  \tfrac{\delta}{\eps}  W_\eps \px W_\eps \\ 
    &\quad \underbrace{+ \tfrac{\delta}{\eps} \int_x^\infty \exp\left(\tfrac{x-y}{\eps}\right) (\py W^\eps)^2 \dd y}_{\ge 0}  +  \tfrac{\delta}{\eps^2} \int_x^\infty \exp\left(\tfrac{x-y}{\eps}\right)   \rho_\eps  \py W_\eps \dd y\\
    & \ge  - V(W_\eps)\pxx W_\eps +\delta (\px W_\eps)^2 - \tfrac{\delta}{\eps} W_\eps \px W_\eps  + \tfrac{\delta}{\eps^2} m  \int_x^\infty \exp\left(\tfrac{x-y}{\eps}\right)   \rho_\eps   \dd y
    \\ &= - V(W_\eps)\pxx W_\eps +\delta (\px W_\eps)^2 - \tfrac{\delta}{\eps} W_\eps \px W_\eps   + \tfrac{\delta}{\eps} m W_\eps.
\end{align*}
We fix $\bar x \in \R$ such that $m(t) = \px W_{\eps}(t,\bar x)$ (we then know that \(\partial_{xx}^{2}W_{\eps}(t, \bar x)=0\)) and evaluate the previous expression at $x = \bar x$. We get  
\begin{align*}
    \tfrac{\d}{\d t} m(t) &\ge  \delta m(t)^2 - \tfrac{\delta}{\eps}   W_\eps(t,\bar x) m(t)  +\tfrac{\delta}{\eps} m(t) W_\eps(t,\bar x)  = \delta m(t)^{2}, \qquad t >0.
\end{align*}
\textbf{Conclusion.} In both cases, we arrive at the Riccati-type differential inequality 
\begin{align*}
    \tfrac{\d}{\d t} m(t) \ge \kappa m^2(t), \qquad t >0 
\end{align*}
(with $\kappa := (\kappa_1 - \kappa_2)$ or $\kappa := \delta$, respectively), which yields 
\begin{align*}
    \frac{W_{\eps}(t,x)-W_\eps(t,y)}{x-y} = \frac{1}{x-y}\int_y^x \px W_{\eps}(t,\xi) \dd \xi \ge - \frac{1}{ \kappa t}, \qquad t >0, \  x,y \in \R, \  x \neq y. 
\end{align*}
\end{proof}

\subsection{Ole\u{\i}nik-type estimate for \texorpdfstring{$V'(W_\eps) W_\eps \partial_x W_\eps$}{V'(W)WW'}}
\label{ssec:proof-2}

The basic idea underpinning the proof of the Ole\u{\i}nik inequality for $g_\eps=V'(W_\eps) W_\eps \partial_x W_\eps$  is to observe that this quantity satisfies the equation
\begin{align*}
    \partial_t g_\eps = &~  (V''(W_\eps)W_\eps + V'(W_\eps))\partial_x W_\eps \partial_t W_\eps + V'(W_\eps)W_\eps \partial^2_{tx} W_\eps.
\end{align*}

\begin{proof}[Proof of Theorem \ref{th:og}]
Owing to Lemma \ref{theo:stability}, it suffices to prove the statement for initial data $\rho_{0} \in \mathcal D \cap C^2(\R)$ and therefore for solutions $\rho_{\e} \in C^2([0,T]\times \R)$.  The set $\mathcal D$ has been defined in~\eqref{e:setD}.
	
	For the sake of brevity, we set $z_\eps := \partial_x W_\eps$. By differentiating \eqref{e:nle*} with respect to $x$, we obtain the following  equation 	for $z_\eps$:	
	\be \label{e:nlez}
	\partial_t z_\eps = - V(W_\eps)\partial_x z - V'(W_\eps)z^2 - g_\eps\ast \partial_x \eta_\eps, \quad (t,x) \in (0,T)\times\R.
	\eq
	From \eqref{e:nle*}, \eqref{e:nlez}, and the fact that 
	\be \label{e:ker_relation}
	\partial_x \eta_\eps = \frac{1}{\eps}\left( \eta_\eps - \delta_0\right), 
	\eq
where $\eta_\eps$ is the same as in~\eqref{e:defetaeps},
	we get
	\be \label{e:g_t}
	\begin{split}
		\partial_t g_\eps = &~  (V''(W_\eps)W_\eps + V'(W_\eps))z_\eps \partial_t W_\eps + V'(W_\eps)W_\eps \partial_t z_\eps \\
		=&~ h_\eps z_\eps \big(-V(W_\eps)z_\eps - g_\eps \ast \eta_\eps\big) + V'(W_\eps)W_\eps \left(-V(W_\eps) \partial_x z_\eps - V'(W_\eps)z_\eps^2-\frac{1}{\eps}\left( g_\eps \ast \eta_\eps -g_\eps\right)\right),
	\end{split}
	\eq
	where 
	\be
	h_\eps := V''(W_\eps)W_\eps + V'(W_\eps),
	\eq
	and
	\be \label{e:g_x}
	\partial_x g_\eps = h_\eps z_\eps^2+ V'(W_\eps)W_\eps \partial_x z_\eps.
	\eq
	We now separately consider two cases:
	\begin{enumerate}
		\item[1.] for every $t \in [0,T]$, there exists $x\in \R$ such that $g_\e(t,x) > 0$;
		\item[2.] there exists $t \in [0,T]$ such that $g_\e(t,x)\le 0$ for every $x\in \R$.
	\end{enumerate}
	\textbf{Case 1.} Owing to Lemma~\ref{theo:stability}, we can assume, with no loss of generality, that, for every $\bar t >0$, we have $\rho_{\eps} (\bar t,\cdot) \in \D \cap C^2(\R)$ and hence $W_\eps (\bar t,\cdot) \in \D \cap C^2(\R)$.
    For every $\bar t \in [0,T)$, there exists a maximum point $\bar x$ of $g_\e(\bar t,\cdot)$.
	In particular, $\partial_x g_\eps(\bar t,\bar x) =0$; by \eqref{e:g_x}, we have
	\be \label{e:z_barx}
	\partial_x z_\eps(\bar t, \bar x) = -\frac{h_\eps}{V'(W_\eps)W_\eps}z_\eps^2(\bar t,\bar x).
	\eq
Evaluating \eqref{e:g_t} at $(\bar t,\bar x)$, we get 
	\be \label{e:I-II-III}
	\begin{split}
		\partial_t g_\eps(\bar t,\bar x) = &~  \left( -h_\eps z_\eps g_\eps \ast \eta_\eps - (V'(W_\eps))^2W_\eps z_\eps^2 - \frac{V'(W_\eps)W_\eps}{\eps}\left( g_\eps \ast \eta_\eps - g_\eps \right)  \right) (\bar t,\bar x) \\
		=: &~ \text{I} + \text{II} + \text{III}.
	\end{split}
	\eq
	We observe that $\text{III} \le 0$ since $V'\le 0$, $W_\eps \ge 0$, and $\bar x$ is a maximum point of $g_\eps (\bar t,\cdot)$.
		Moreover, by using the definition of $g_\e$ and the maximum principle, we get
	\be \label{e:II}
	\text{II} = - \frac{g_\e^2}{W_\e} \le - \frac{1}{\|\rho_{0}\|_{L^\infty(\R)}} g_\e^2.
	\eq
		The term $\text{I}$ is more delicate and can be controlled using the assumptions \eqref{ass:ob2} or \eqref{ass:ob3}.\\
	\textbf{Case  1a.} Under the assumption \eqref{ass:ob2}, we have $h_\eps \le 0$. Therefore, if $g_\eps \ast \eta_\eps (\bar t,\bar x)\ge 0$, then $\text{I}\le 0$.
	Otherwise, let us assume that  $g_\eps \ast \eta_\eps (\bar t,\bar x)< 0$:
	since $z_\e = \rho_{\e} \ast \partial_x \eta_\e$ then by recalling \eqref{e:ker_relation} we arrive at
	\be 
	|z_\e| = \left| \frac{1}{\e}\left( \rho_\eps \ast \eta_\e - \rho_\eps \right) \right| \le \frac{\Osc \, \rho_\eps}{\e}
	\eq
	and therefore
    \be 
	| h_\e z_\e g_\eps \ast \eta_\eps (\bar t,\bar x)| = |\text{I}| \le \frac{\Osc \, \rho_\eps}{\e} \left|  h_\e g_\eps \ast \eta_\eps (\bar t,\bar x)\right| \le \frac{|V'(W_\e) W_\e| }{\e} |g_\eps \ast \eta_\eps (\bar t,\bar x)| \le |\text{III}|,
	\eq 
	where we used \eqref{ass:ob2} and $h_\e\le 0$ in the second inequality and $g_\eps \ast \eta_\eps (\bar t,\bar x)< 0$ in the last inequality.  In particular, this shows
	\be \label{e:est_g}
	\partial_tg_\e(\bar t,\bar x) \le - \frac{1}{\|\rho_{0}\|_{L^\infty(\R)}} g_\e^2(\bar t,\bar x),
	\eq
	which, by comparison, yields the desired claim.\\
	\textbf{Case  1b.} Under the assumption \eqref{ass:ob3}, we have $h_\eps \ge 0$. 
	In case  $g_\eps \ast \eta_\eps (\bar t,\bar x)\le 0$, then $\text{I}\le 0$. We then focus on the case  $g_\eps \ast \eta_\eps (\bar t,\bar x)> 0$. 
	Since $\bar x$ is a maximum point for $g_\e(\bar t,\cdot)$, then  $g_\eps \ast \eta_\eps (\bar t,\bar x) \le g_\eps(\bar t,\bar x)$; hence
	\begin{equation*}
		\begin{split}
			\text{I} + \text{II} \le & - \big[ h_\e z_\e g_\e  + (V'(W_\e))^2 W_\e z_\e^2 \big] (\bar t, \bar x) \\
			= & -W_\e V'(W_\e) z_\e^2 (V''(W_\e)W_\e + 2 V'(W_\e)) (\bar t, \bar x) \\
			\le &~ - \kappa_1 W_\e (V'(W_\e))^2 z_\e^2 (\bar t, \bar x) \\
			= &~ -\frac{\kappa_1}{W_\e}g_\e (\bar t,\bar x)^2 \\
			\le &~ - \frac{\kappa_1}{\|\rho_{0}\|_{L^\infty(\R)}}g_\e (\bar t,\bar x)^2,
		\end{split}
	\end{equation*}
	where, in the second inequality, we used \eqref{ass:ob3}. 
 This establishes 	\eqref{e:est_g} which, by comparison, yields \eqref{e:claim}.\\
	\textbf{Case 2.} We define $\bar t \in [0, T]$ by setting 
	\be 
	\bar t : = \inf \{ t \in [0, T]:  \;  g_\eps(t,x)\le 0 \; \text{for every $x\in \R$} \}.
	\eq
	Assuming that $\bar t>0$, we can apply the same argument as in \textbf{Case 1} on the interval $[0, \bar t)$. Since $t\mapsto \Lip^-\rho_{\e}(t)$ is a continuous function, then also $t \mapsto \max g_\e(t,\cdot)$ is continuous and this establishes ~\eqref{e:claim} on $[0, \bar t]$.
	Note that $ g_\eps(t,x)\le 0$ for every $x\in \R$ if and only if  $\rho_{\e}(t,\cdot)$ is non-decreasing. 
	Therefore, since~\eqref{eq:cl} preserves the monotonicity of the initial datum (see~\cite{MR3447130,MR3670045}), then, for every $t \in (\bar t, T]$, $\rho_{\e}(t, \cdot)$ is a monotone non-decreasing function, that is $g_\eps(t) \leq 0$. If $\bar t=0$, then we can directly apply the argument for the preservation of monotonicity. This concludes the proof.
\end{proof}

\begin{remark}[The Greenberg model]\label{rk:greenb}
	Let us consider the velocity function $V(\xi)= v_0 \ln \left( \rho_{\max}/{\xi} \right)$ with $v_0 >0$ and $\rho_{\max} >0$, which corresponds to a traffic model proposed by Greenberg and supported by experimental data (see \cite[Chapter 3, Eq. (3.1.4)]{MR2328174}). Formally, an Ole\u{\i}nik-type estimate still holds: indeed, going back to \eqref{e:I-II-III}, we get 
	$h_\eps \equiv 0$; thus $\text{I}=0$ 
	therefore, since $\text{III}\le 0$ and \eqref{e:II}, it follows from \eqref{e:I-II-III} that
	\begin{align*}
	\partial_tg_\e(\bar t,\bar x) \le - \frac{1}{\|\rho_{0}\|_{L^\infty(\R)}} g_\e^2(\bar t,\bar x),
	\end{align*}
	which, by comparison, implies \eqref{e:claim}. Assuming that the initial density is bounded away from zero, this remark can be made rigorous.
\end{remark}

\section{Proof of the convergence in the nonlocal-to-local singular limit}
\label{sec:proof-limit}

As a first step towards the proof of Theorem \ref{T_main}, we point out that Theorem \ref{th:o} implies a uniform BV estimate (see \cite[Eq. (4.3)]{MR1618393} and \cite[Lemma 2.2 (ii) \& Remark 2.3]{MR2165401}) and, thus, compactness of $\{W_\eps\}_{\eps>0}$ for $t>0$. 

\begin{lemma}[BV-regularization and compactness]\label{cor:bv}
	Let us assume that  \eqref{eq:ol} holds. Then the solution $W_\eps(t,\cdot)$ of \eqref{eq:Wt} belongs to $\mathrm{BV}_{\mathrm{loc}}(\R)$ for every $t >0$ uniformly with respect to $\eps >0$: namely, for every compact interval $K \Subset \R$, 
	\begin{align}\label{eq:bv-bj}
	|W_\eps(t,\cdot)|_{\mathrm{TV}(K)} \le  2 \left( \frac{|K|}{2t} + \|W_\eps(t,\cdot)\|_{L^\infty(K)}\right).
	\end{align}
	This implies that the set \(\{W_{\eps}\}_{\eps>0}\) is compactly embedded into $L^1_{\loc}((0,T)\times \R)$.
\end{lemma}

\begin{proof}
	The claim in \eqref{eq:bv-bj} is contained in \cite[Eq. (4.3)]{MR1618393} or \cite[Lemma 2.2 (ii) \& Remark 2.3]{MR2165401}. The second one follows by arguing as in \cite[Theorem 4.1]{2012.13203}.
\end{proof}

With Lemma \ref{cor:bv} in hand, we can directly establish Corollary~\ref{T_main} under the assumptions~\eqref{ass:v-conv-more} or \eqref{ass:linear} -- i.e. using the Ole\u{\i}nik inequality from Theorem \ref{th:o} -- by arguing similarly as in \cite[Corollary~4.1 \& Theorem 4.2]{2012.13203}. In fact, more simply, to prove that the limit point of $\{W_\eps\}_{\eps>0}$ is an entropy admissible solution of the local conservation law \eqref{eq:cll}, it suffices to pass to the limit pointwise in \eqref{eq:ol}.

The proof of Theorem \ref{T_main} under the assumptions \eqref{ass:ob2} or \eqref{ass:ob3} -- i.e., using the Ole\u{\i}nik inequality from Theorem \ref{th:og} -- is somehow more delicate. Indeed, we cannot directly deduce a uniform $\mathrm{BV}$ bound on $\{W_\eps\}_{\eps>0}$. In Lemma \ref{l:weak} below, we rather show that $W_\eps^2$ is  equi-bounded in $\BV_{\loc}((0,T)\times \R)$ and, therefore, that the family $W_\e$ is precompact in $L^1_{\loc}((0,T)\times \R)$ and that limit points $W$ of $W_\e$ as $\e \to 0$ are weak solutions of \eqref{eq:cll}. The fact that the limit point of $\{W_\eps\}_{\eps>0}$ so constructed is an entropy-admissible solution of the local conservation law is already known from \cite{MR4283539}. In Lemma \ref{l:strong}, we present, however, an independent proof. We point out that the Ole\u{\i}nik-type inequality for $W_\e^2$ rules out the presence of non-entropic shocks in the limit $W$. When $W$ does not have bounded variation it is not trivial to deduce that it is in fact the entropy-admissible solution: we achieve this by exploiting the recent results of \cite{MR4054358,marconi_rectifiability} on Besov regularity and on the structure of solutions of conservation laws with finite entropy production. This seems to be of independent interest. 

Finally, we need to show that $\rho_{\e}$ converges to the same limit as $W_\eps$. If we have a total variation bound on $W_\eps$, this follows immediately from the identity \eqref{eq:WWxq}. In case the bound holds only for $W_\eps^2$, a more subtle analysis is needed, which we perform in Lemma~\ref{l:u}.

\begin{lemma}[Precompactness in $L^1$]\label{l:weak}
Let us assume that \eqref{e:claim} and  $V'\leq - \kappa_2$ hold. Then the sequence $\{W_\eps\}_{\eps>0}$ is precompact in $L^1_{\loc}((0,T)\times \R)$ and every accumulation point of $W_\e$ is a weak solution of \eqref{eq:cll}.  
\end{lemma}
\begin{proof}
	\textbf{Step 1:} \emph{Precompactness of $W_\e$.}
	Since $V' < -\kappa_2$, then, from $g_\e(t, \cdot) \le \frac{1}{\kappa t}$, we deduce 
	\be \label{e:est_W2}
\partial_x W^2_\e(t,\cdot) \le \frac{2}{\kappa_2 \kappa t}
	\eq
	and   
	\[ 
	\partial_t W^2_\e(t,\cdot) = -V(W_\e)\partial_x W_\e^2 - 2 W_\e g_\e \ast \eta_\e \ge - \frac{2 V(0) +2  \max \rho_{0}}{\kappa \kappa_2 t}
	\]
 for $t>0$. 	In particular, this yields that $W^2_\e$ is equi-bounded in $\BV_{\loc}((0,T)\times \R)$. By Helly's compactness theorem, there is a subsequence $W^2_{\e_k}$ which converges a.e. to some function $W^2$. Therefore $W_{\e_k}$ converges to $W$ a.e. and, by Lebesgue's Dominated Convergence Theorem, $W_{\e_k} \to W$ in $L^1_{\loc}((0,T)\times \R)$. \\
	\textbf{Step 2:} \emph{$W$ is a weak solution of \eqref{eq:cll}}.
	By \eqref{e:nle*}, it suffices to show that
	$g_\e - g_\e \ast \eta_\e \to 0$ in $\mathcal D'([0,T)\times \R)$.
		Let us first fix $\varphi \in C^\infty_c((0,T)\times \R)$, then
	\[
	\iint_{(0,T)\times \R} \varphi ( g_\e - g_\e \ast \eta_\e) \dd x \dd t = \iint_{(0,T)\times \R} \varphi g_\e \ast ( \delta_0 - \eta_\e) \dd x \dd t 
	= \iint_{(0,T)\times \R} \varphi \ast ( \delta_0 - \tilde \eta_\e) g_\e \dd x \dd t,
	\]
	where $\tilde \eta_\e ( x) := \eta_\e (-x)$. Since $\varphi(t,\cdot) \ast ( \delta_0 - \tilde \eta_\e)$ converges uniformly to 0 and decays exponentially in space uniformly in $\e$ and 
	\[
	\int_{-L}^L |g_\e(t,x)|\dd x\le \|V'\|_{L^\infty(\R)} \mathrm{TV}_{[-L,L]}W_\e^2(t,\cdot) 
	\]
	grows at most linearly in $L$ owing to \eqref{e:est_W2}, then for every $\varphi \in  C^\infty_c((0,T)\times \R)$ we have
	\[
	\lim_{\e\to 0}\iint_{(0,T)\times \R} \varphi ( g_\e - g_\e \ast \eta_\e) \dd x \dd t = 0.
	\]
	We now fix $\varphi \in  C^\infty_c([0,T)\times \R)$; since $\rho_{\e}$ solves \eqref{eq:cl}, then the map
	\[
	t\mapsto \int_\R \rho_{\e}(t,x)\varphi(t,x) \dd x
	\] 
	is Lipschitz continuous with respect to $t$ uniformly with respect to $\eps$ on $[0,T)$. Therefore, the same is true if we replace $\rho_{\e}$ by $W_\e:= \rho_{\e} \ast \eta_\e$.
	In particular, by \eqref{e:nle*}, we have that 
	\[
	t \mapsto \int_\R (g_\e -g_\e \ast \eta_\e) \varphi(t,x) \dd x
	\]
	is Lipschitz continuous with respect to $t$ uniformly with respect to $\e$ on $[0,T)$. Hence $g_\e -g_\e \ast \eta_\e \to 0$ in $\mathcal D'([0,T)\times \R)$.
\end{proof}

\begin{lemma}[Entropy admissibility of the limit point]\label{l:strong}
Let us assume that \eqref{e:claim} holds.	If $W$ is an accumulation  point of $W_\e$,  then $W$ is the entropy admissible solution of \eqref{eq:cll}.
\end{lemma}
\begin{proof}
	
	We already know from Lemma \ref{l:weak} that $W$ is a weak solution of \eqref{eq:cll}. 
	Moreover, since $W$ is a limit point of $W_\e$, then $W^2\in \BV_{\loc}((0,T)\times \R)$. 
	We check that this implies $W \in B^{1/3,3}_{\infty,\loc}((0,T)\times \R)$:
	indeed, given $\Omega$ compactly contained in $(0,T)\times \R$ and $h \in \R^2$ sufficiently small, we have
	\[
	\int_\Omega |D_h W_\e|^3 \dd x\le \| \rho_{0}\|_{L^\infty(\R)} \int_\Omega |D_h W_\e|^2 \dd x\le 
 \| \rho_{0}\|_{L^\infty(\R)} \int_{\Omega_h} |D_h W_\e^2| \le  \| \rho_{0}\|_{L^\infty(\R)} |h| \mathrm{TV}_{\Omega_h} W_\e^2,
	\]
	where $\Omega_h:= \{ (t,x) \in (0,T)\times \R: \mathrm{dist}(x,\Omega)\le |h|\}$ and we used $0\le W_\e\le \| \rho_{0}\|_{L^\infty(\R)}$.
	Weak solutions $W$ to Burgers equation belonging to $B^{1/3,3}_{\infty,\loc}((0,T)\times \R)$ enjoy a kinetic formulation (see \cite[Theorem 2.6]{MR4054358}) and for every weak solution enjoying a kinetic formulation there are countably many Lipschitz continuous curves $\gamma_n:[0,T)\to \R$ such that for every entropy-entropy flux pair $(e,f)$ and every $\varphi \in C^\infty_c((0,T)\times \R)$ we have 
	\be\label{e:dissipation}
	\begin{split}
		\iint_{(0,T)\times \R} & \left(e(W) \partial_t \varphi +  f(W) \partial_x \varphi\right)  \dd x \dd t  \\
		= &\sum_{n=1}^\infty \int_0^T \varphi \left[ f(W^+)- f(W^-) - \dot \gamma_n(t) (e(W^+) - e(W^-)) \right] (t,\gamma_n(t))\dd t,
	\end{split}
	\eq
	where $W^\pm$ denotes the traces of $W$ along $\gamma_n$ (see \cite{marconi_rectifiability}). The uniform one-side bound on $g_\e$ proven in Proposition \ref{th:og} implies that for 
	every $n$ and a.e. $t\in (0,T)$ we have  $W^+(t,\gamma_n(t)+) \ge W^-(t,\gamma_n(t)-)$. 
	Since $u\mapsto uV(u)$ is concave, then it is well-known that the shocks with $W^+ \ge W^-$ are entropic, namely for every convex entropy $e$ and every $W^- \le W^+$ we have 
	\[ 
	f(W^+)- f(W^-) - \dot \gamma_n(t) (e(W^+) - e(W^-))  \ge 0.
	\]
	In particular, by \eqref{e:dissipation}, we have that $W$ is the entropy solution of \eqref{eq:cll}.
\end{proof}

\begin{lemma}[Convergence of $\rho_\eps$]\label{l:u}
Let us assume that \eqref{e:claim} holds. Then the functions $\rho_{\eps}$ converge to $W$ in $L^1_{\loc}((0,T)\times \R)$ as $\eps \to 0$.
\end{lemma}
\begin{proof}
	Owing to the specific choice of the kernel $\eta_\e$, we have the relation
	\be \label{e:uW}
	\rho_{\eps} = W_\eps - \eps \px W_\eps.
	\eq
	Therefore, by \eqref{e:est_W2}, we deduce 
	\[
	W_{\e}^2 - W_{\e} \rho_{\e} = W_{\e} ( W_{\e} - \rho_{\e}) = \e W_{\e} z_{\e} = \frac{\e}{2}\partial_x W_{\e}^2 \to 0 \quad \mbox{in }L^1_{\loc}((0,T)\times \R),
	\]
	so that there is a sequence $\e_k\to 0$ such that $\rho_{\e_k}$ converges to $W$ a.e. in the set $\{W\ne 0\}$.
	
	We now discuss the convergence on the set $\{ W= 0 \}$.
	Given $\bar t, L>0$, let us define
	\[
	A(\bar t, L):= \{ (t,x) \in (0,T)\times \R : x \in (-L - V_{\max}(\bar t - t), L + V_{\max}(\bar t -t))\},
	\]
	where $V_{\max} = V(0)= \max V$. 
	Up to removing a negligible set of values for $\bar t$ and $L$, we can assume that $\mathcal H^1$-a.e.\ point in $\partial A(\bar t, L) \cap (0,T)\times \R$ is a Lebesgue point of $W_{\e_k}$ and $\rho_{\e_k}$ for every $k\in \N$.
	Taking a further subsequence of $\e_k$, which we do not rename, we can assume that $W_{\e_k}$ converges to $W$ a.e.\ in $(0,T)\times \R$.
	
	Given $h>0$, let us consider an increasing function $\chi_h \in C^\infty(\R)$ such that
	\[
	\chi_h(x) = 
	\begin{cases}
	1 & \mbox{if }x\ge h,\\
	0 & \mbox{if }x \le 0,
	\end{cases}
	\]
	and the approximation $\varphi_h$ of the characteristic function of $A(\bar t,L)$ defined by
	\[
	\varphi_h(t,x) = \chi_h(\bar t - t) \chi_h ( x + L +V_{\max}(\bar t - t)) \chi_h ( L + V_{\max}(\bar t -t) - x).
	\]
	Testing \eqref{eq:cl} with $\varphi_h$ and letting $h\to 0$, we get
	\be \label{e:balance1}
	\int_{-L - V_{\max}\bar t }^{L + V_{\max}\bar t } \rho_{0}(x) \dd x - \int_{-L}^L \rho_{\e} (\bar t, x) \dd x= \int_0^{\bar t} \mathcal F^+(\rho_{\e})(t) \dd t +  \int_0^{\bar t} \mathcal F^-(\rho_{\e})(t) \dd t, 
	\eq
	where
	\[
	\begin{split}
	\mathcal F^+(\rho_{\e})(t) := &\left(\rho_{\e} V(W_\e) + V_{\max} \rho_{\e}\right) (t, L + V_{\max}(\bar t - t)), \\
	\mathcal F^-(\rho_{\e})(t) := & \left(-\rho_{\e} V(W_\e) + V_{\max} \rho_{\e}\right) (t, -L - V_{\max}(\bar t - t))
	\end{split}
	\]
	are the exiting fluxes of the quantity $\rho_{\e}$ across the lateral boundaries of $A(\bar t,L)$.
	Since $\rho_{\e_k}\to W$ in the set $\{W\ne 0\}$ and $\rho_{\e_k}\ge 0$, then
	\be \label{e:limsup}
	\limsup_{k\to \infty} \int_{-L - V_{\max}\bar t }^{L + V_{\max}\bar t } \rho_{0}(x) \dd x - \int_{-L}^L \rho_{\e_k} (\bar t, x) \dd x\le  \int_{-L - V_{\max}\bar t }^{L + V_{\max}\bar t } \rho_{0}(x) \dd x - \int_{-L}^L W  (\bar t, x) \dd x.
	\eq
Similarly, observing that $\xi \mapsto \mathcal F^\pm(\xi)$ is increasing, we have
	\be \label{e:liminf}
	\liminf_{k\to \infty} \int_0^{\bar t} \mathcal F^+(\rho_{\e_k})(t) \dd t +  \int_0^{\bar t} \mathcal F^-(\rho_{\e_k}) \dd t \ge  \int_0^{\bar t} \mathcal F^+(W)(t) \dd t +  \int_0^{\bar t} \mathcal F^-(W)(t) \dd t.
	\eq
	Now let us test \eqref{e:nle*} with $\varphi_h$ and let $\e \to 0$: since $g_\e - g_\e \ast \eta_\e \to 0$ in the sense of distributions on $[0,T)\times \R$, we get
	\[
	\int_{(0,T)\times \R} \Big( W \pt \varphi_h + W V(W) \px \varphi_h \Big) \dd x \dd t + \int_\R \rho_{0}(x) \varphi_h(0,x) \dd x=0.
	\]
	Letting $h\to 0$, we thus obtain
	\be \label{e:balance2}
	\int_{-L - V_{\max}\bar t }^{L + V_{\max}\bar t } \rho_{0}(x)\dd x - \int_{-L}^L W (\bar t, x) \dd x= \int_0^{\bar t} \mathcal F^+(W)(t) \dd t +  \int_0^{\bar t} \mathcal F^-(W)(t) \dd t.
	\eq
	Comparing \eqref{e:balance1} and \eqref{e:balance2}, we get that the two inequalities \eqref{e:limsup}, \eqref{e:liminf} are actually equalities and the liminf and limsup are actually limits. In particular, since $\rho_{\e_k}\ge 0$, it follows from \eqref{e:limsup} and $\rho_{\e_k}\to W$ in $\{W\ne 0\}$ that
	\[ 
	\lim_{k\to \infty} \int_{\{ W= 0\} \cap [-L,L]} \rho_{\e_k}(\bar t,x) \dd x= 0
	\]
	and therefore $\rho_{\e_k}(\bar t) \to W_\eps(\bar t)$ in $L^1_{\loc}(\R)$.
	Since the limit $W$ does not depend on the subsequence $\e_k$ we are considering, we conclude that 
	\[
	\rho_{\e} \to W \qquad \mbox{in }L^1_{\loc}((0,T)\times \R). \qedhere
	\]
\end{proof}
\begin{remark}[Effect of a lower bound on the density]
		The proof of the convergence result is easier and self-contained if we also assume a lower bound on the density: 
  \begin{align}\label{eq:lbd}
      \essinf \rho_{0}  \ge c_0 >0.
  \end{align}
From \eqref{eq:lbd}, we can show 
\begin{align}\label{eq:lbd-e}
\essinf \rho_{\e} \ge \ess \inf \rho_{0}\ge c_0 > 0.
\end{align}
   Let us note that, in this case, the generalized California model and the Greenberg model  mentioned above (which are not Lipschitz continuous at zero density) are well-posed.

From \eqref{e:claim}, \eqref{eq:lbd-e} and the upper bound on $V'\leq - \kappa_2$, we deduce that, for every $t>0$,
			\be 
             \sup_{\mathbb R}  \partial_x W_\e(t, \cdot) \geq - \frac{1}{\kappa \kappa_2 c_0 t}.
   \eq
			This implies that $W_\e \in \BV_{\loc}((0,+\infty)\times \R)$ uniformly with respect to $\e>0$. In particular, let $W$ be an accumulation point of 
			$W_\e$ as $\e \to 0$ in $L^1_\loc((0,+\infty)\times \R)$, then $W$ solves \eqref{eq:cll} and, since it is one-sided Lipschitz continuous, it coincides with the entropy solution $\rho$. 
			
			In order to complete the proof, we only need to show that $\rho_{\e}$ also converges to $\rho$. We follow the argument in \cite{2012.13203}: by \eqref{e:ker_relation} we have
			\[
			\rho_{\e} = W_\e - \e \partial_x W_\e.
			\]
			Being $\partial_x W_\e$ equi-bounded in $L^1_\loc$, the two sequences $\rho_{\e}$ and $W_\e$ converge to the same limit function $\rho$.
\end{remark}

\begin{proof}[Proof of Corollary \ref{T_main}] We proceed according to the following steps. \\
	\textbf{Step 1:} \emph{proof using Theorem \ref{th:o}.}  We assume \eqref{eq:ol} and apply Lemma~\ref{cor:bv} to deduce that \(\{W_{\eps}\}_{\eps >0}\) is compactly embedded in $L^{1}_{\loc}((0,T)\times\R)$. Then, by arguing as in \cite[Corollary~4.1 \& Theorem 4.2]{2012.13203}, we obtain that $W_\eps$ converges to the unique entropy solution of the local conservation law \eqref{eq:cll} and so does $\rho_{\eps}$. We only need to pay extra attention to the fact that the convergence holds on every compact set contained in the open set $t>0$. To this end, given a parameter $n \in \N$ and a test function $\varphi \in C^\infty_c([0,+\infty)\times\R)$, as in~\cite[Corollary 4.1 \& Theorem 4.2]{2012.13203}, by the compactness of $\{\rho_\eps\}_{\eps >0}$ in $L^1_{\loc}((0,T)\times\R)$, we can pass to the limit in the entropy inequality as $\eps \to 0^+$ and deduce 
\begin{align*}
    0 &\le 
    \underbrace{\int_{1/n}^{T}\int_{\R} \big(\eta(\rho(t,x)) \pt \varphi(t,x)  + q(\rho(t,x))\px \varphi(t,x) \big) \dd x \dd t}_{I_{1,n}} \\
      &\qquad + \underbrace{\int_{0}^{1/n}\int_{\R} \big( \bar \eta(t,x) \pt \varphi(t,x) +  \bar q(t,x) \px \varphi(t,x)\big) \dd x \dd t}_{I_{2,n}}  + \int_{\R} \eta(\rho_0(x)) \varphi(0,x) \dd x,
\end{align*}
where $\eta(\rho_\eps) \weaks \bar \eta$ and $q(\rho_\eps) \weakstar \bar q$ in $L^\infty(\R)$ by the uniform $L^\infty$-bound on $\{\rho_\eps\}_{\eps>0}$. By letting $n \to \infty$, we then deduce   
$$0 \le \int_{0}^{T}\int_{\R} \big(\eta(\rho(t,x)) \pt \varphi(t,x) + q(\rho(t,x)) \px \varphi(t,x) \big) \dd x \dd t + \int_{\R} \eta(\rho_0(x)) \varphi(0,x) \dd x,
$$
where we used the fact that $I_{2,n} \to 0$ because of the $L^1$ bound on the integrand. \\
\textbf{Step 2:} \emph{proof using Theorem \ref{th:og}.} We assume \eqref{e:claim}, then the claim follows by combining Lemmas \ref{l:weak}, \ref{l:strong}, and \ref{l:u}, and the computation above.
\end{proof}

\section{Numerical experiments}
\label{sec:numerics}

In this section, we illustrate the results of Theorem~\ref{th:o} and Theorem~\ref{th:og} with some numerical simulations. For the nonlocal problem, we rely on a non-dissipative solver based on characteristics (see \cite{pflug2} for further details). In particular, we consider the Greenshields velocity function $V(\xi) = 1-\xi$; in Figure \ref{fig:sim1} and Figure \ref{fig:sim2} we show the behavior of $t \mapsto \partial_x W_\eps(t,\cdot)$ for two types of initial data, continuous (Figure  \ref{fig:sim1}) and with a jump discontinuity (Figure \ref{fig:sim2}). We present simulations for both the exponential kernel (top row of Figures \ref{fig:sim1} and \ref{fig:sim2}) and for a piecewise constant kernel $\eta := \eps^{-1}\mathds{1}_{(0,\eps)}$ (bottom row of Figures \ref{fig:sim1} and \ref{fig:sim2}) which is not covered by the results of the present paper; the same result appears to hold in this case too. Finally, in Figure \ref{fig:sim3}
we highlight the $\mathrm{BV}$-regularization effect on $W$ provided by the Ole\u{\i}nik inequality. 

\begin{figure}[ht]
    \centering
    \includegraphics[scale=0.75,trim=3 29 20 0]{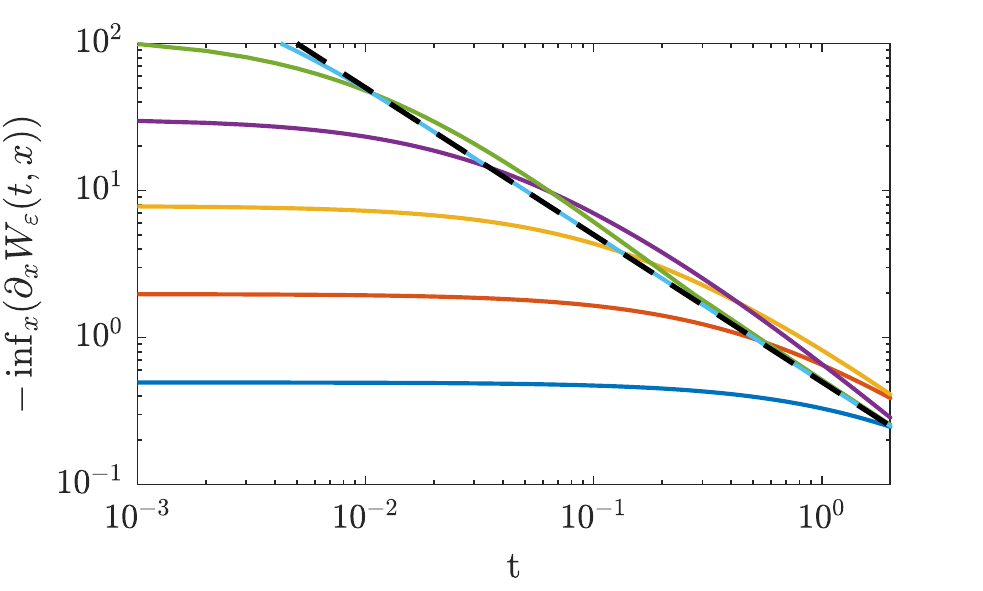}
\includegraphics[scale=0.75,trim=0 29 20 0]{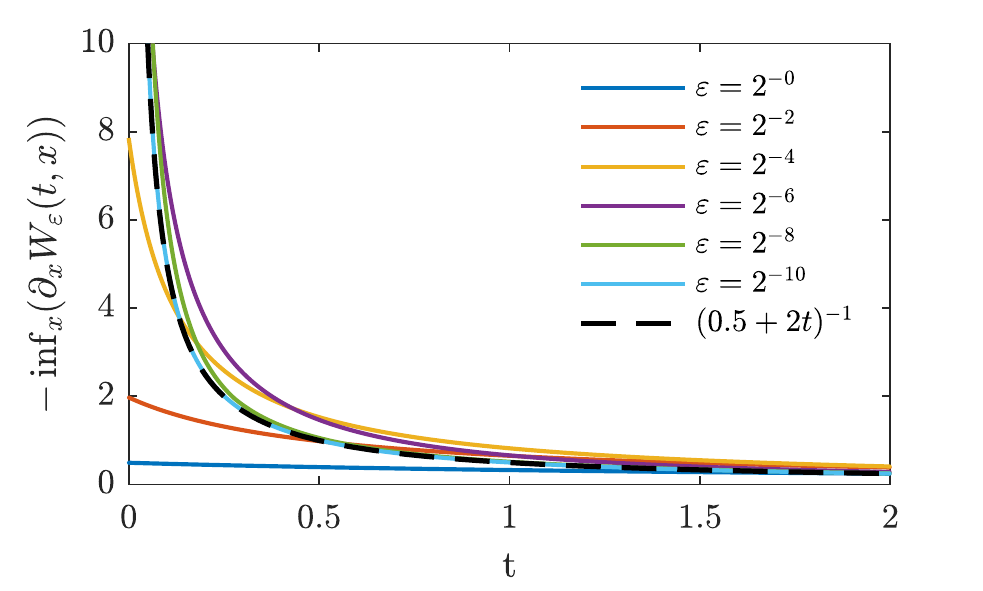}

\includegraphics[scale=0.75,trim=0 0 20 0]{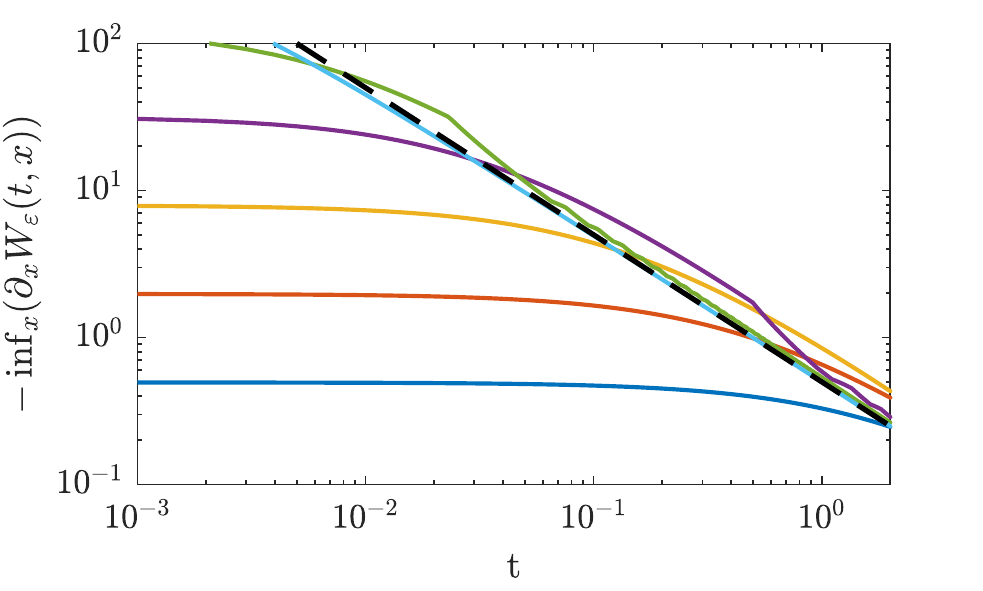}
\includegraphics[scale=0.75,trim=0 0 20 0]{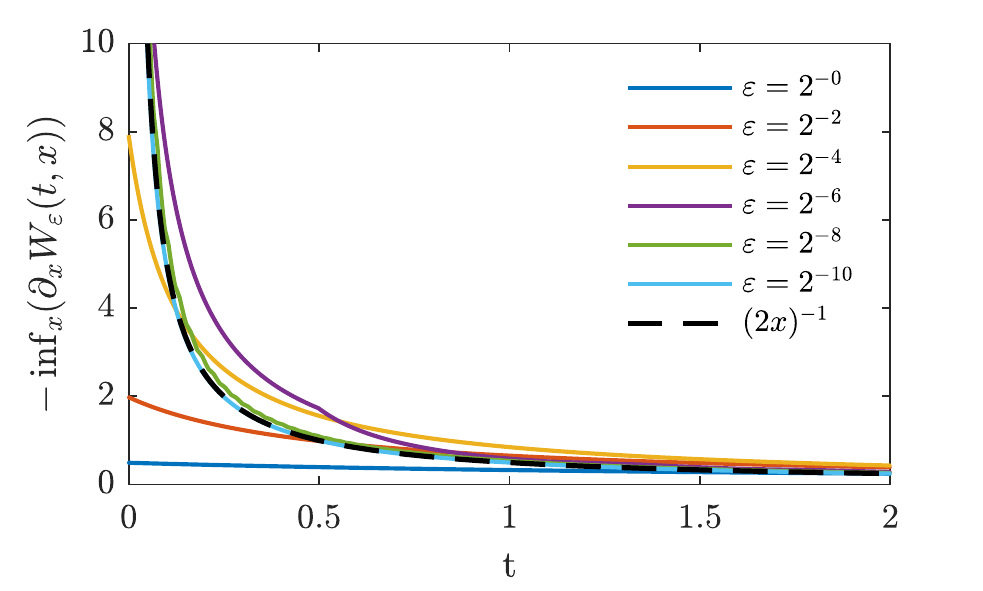}
    \caption{Illustration of $-\inf \partial_x W_\eps$. Simulations for the initial datum $\rho_0 := \frac{1}{2}\mathds{1}_{(-0.5,0.5)}$ and velocity $V(\xi) = 1-\xi$. \textsc{Top row}: kernel $\eta(\cdot) := \eps^{-1}\exp(-\cdot \eps^{-1})$. \textsc{Bottom row}: kernel $\eta := \eps^{-1}\mathds{1}_{(0,\eps)}$.}
    \label{fig:sim1}
\end{figure}

\begin{figure}[ht]
    \centering

\includegraphics[scale=0.75,trim=3 29 20 0]{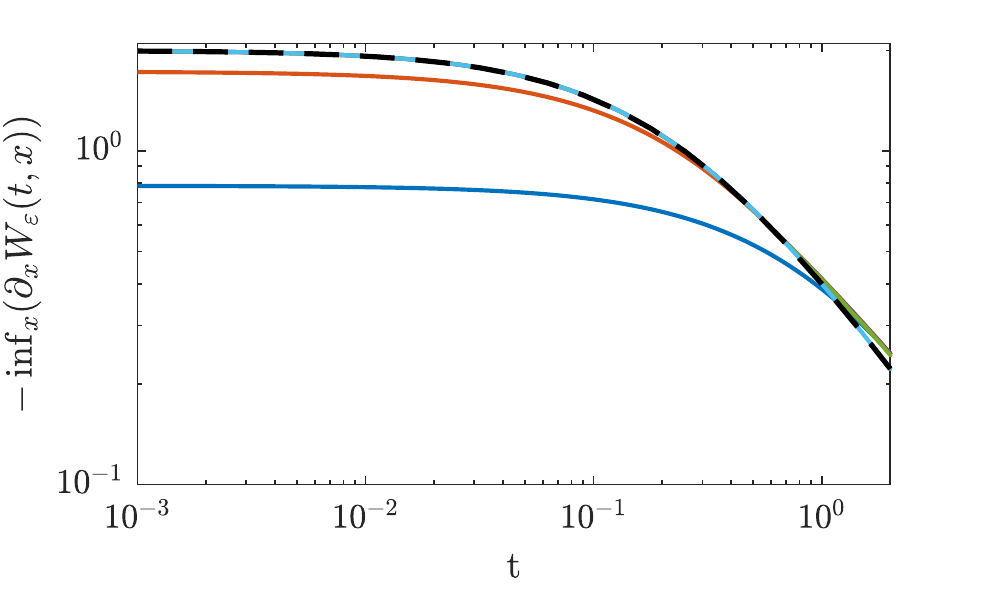}
\includegraphics[scale=0.75,trim=0 29 20 0]{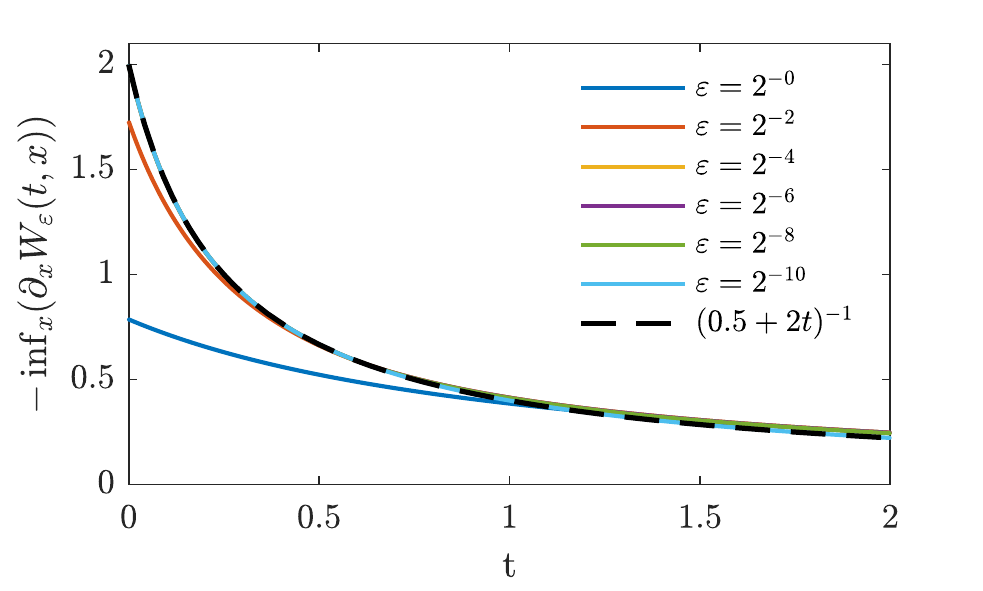}

\includegraphics[scale=0.75,trim=0 0 20 0]{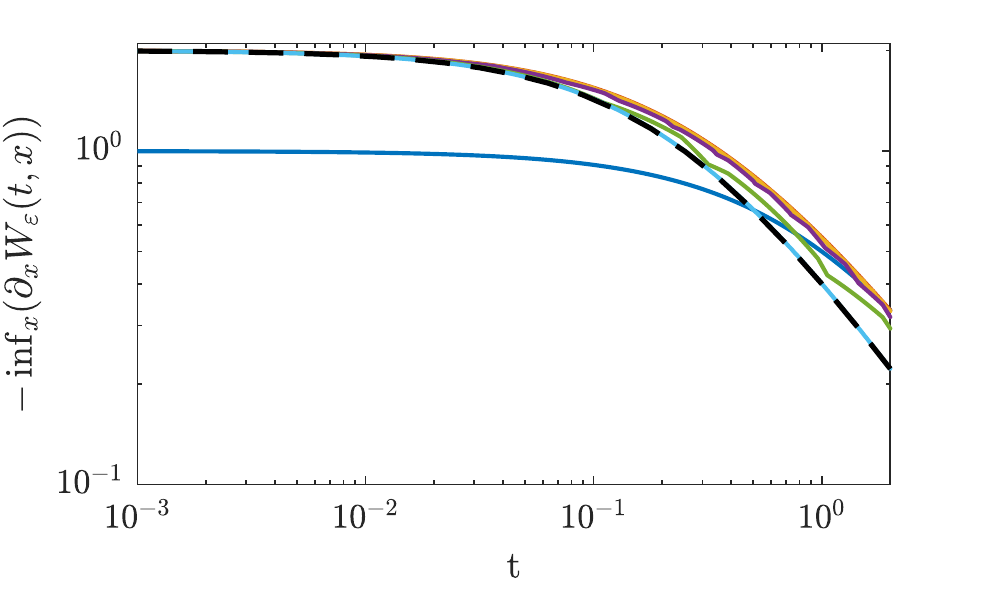}
\includegraphics[scale=0.75,trim=0 0 20 0]{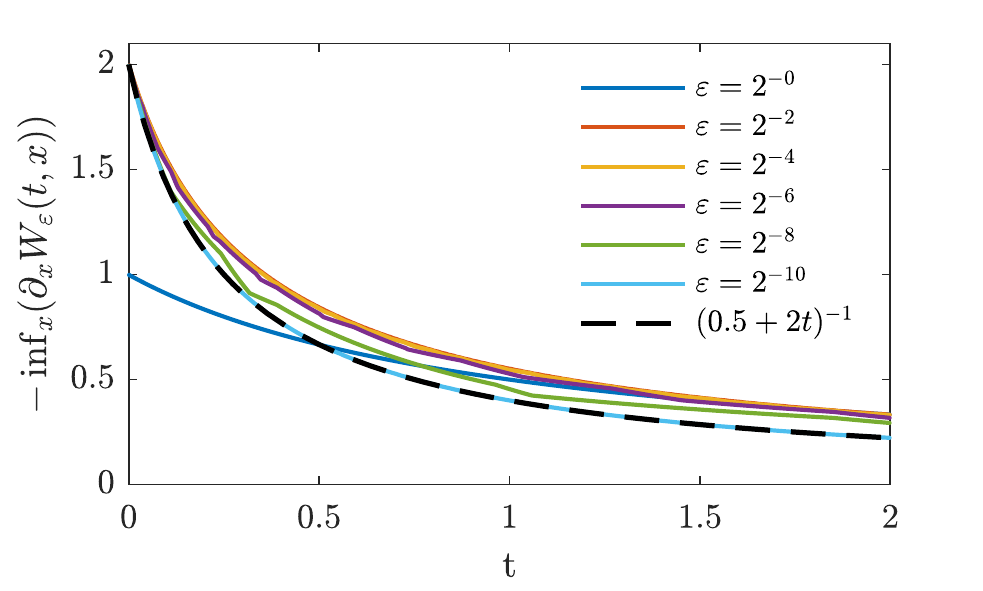}
   \caption{Illustration of $-\inf_{x\in \R} \partial_x W_\eps(t,x)$. Simulations for the initial datum $\rho_0(\cdot) := (1-2|\cdot|)\mathds{1}_{(-0.5,0.5)}$ and velocity $V(\xi) = 1-\xi$.  \textsc{Top row}: kernel $\eta(\cdot) := \eps^{-1}\exp(-\cdot \eps^{-1})$. \textsc{Bottom row}: kernel $\eta (\cdot):= \eps^{-1}\mathds{1}_{(0,\eps)}(\cdot)$}.
    \label{fig:sim2}
\end{figure}

\begin{figure}[ht]
    \centering
    \includegraphics[scale=0.75,trim=3 0 20 0]{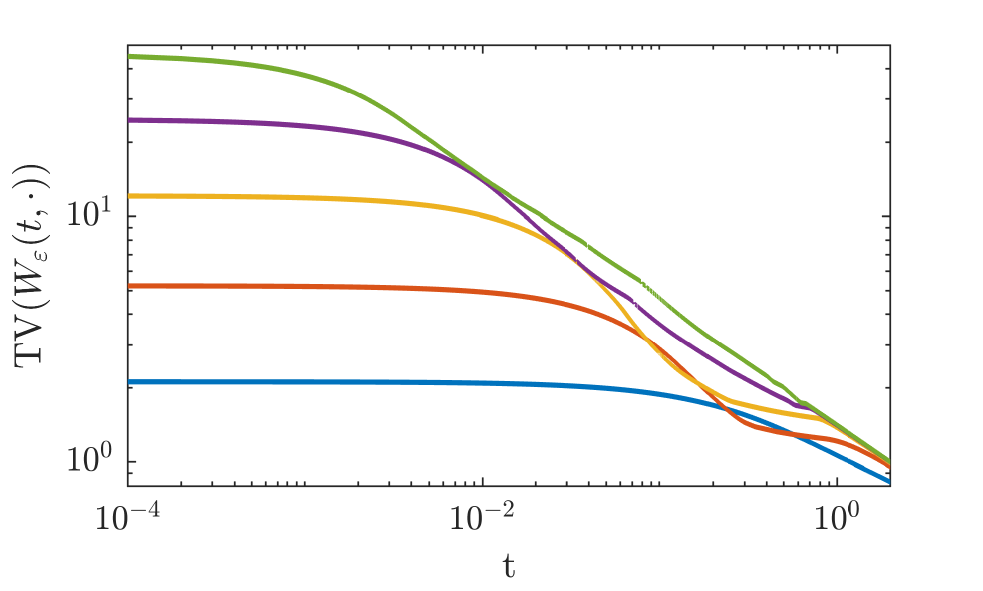}
    \includegraphics[scale=0.75,trim=0 0 20 0]{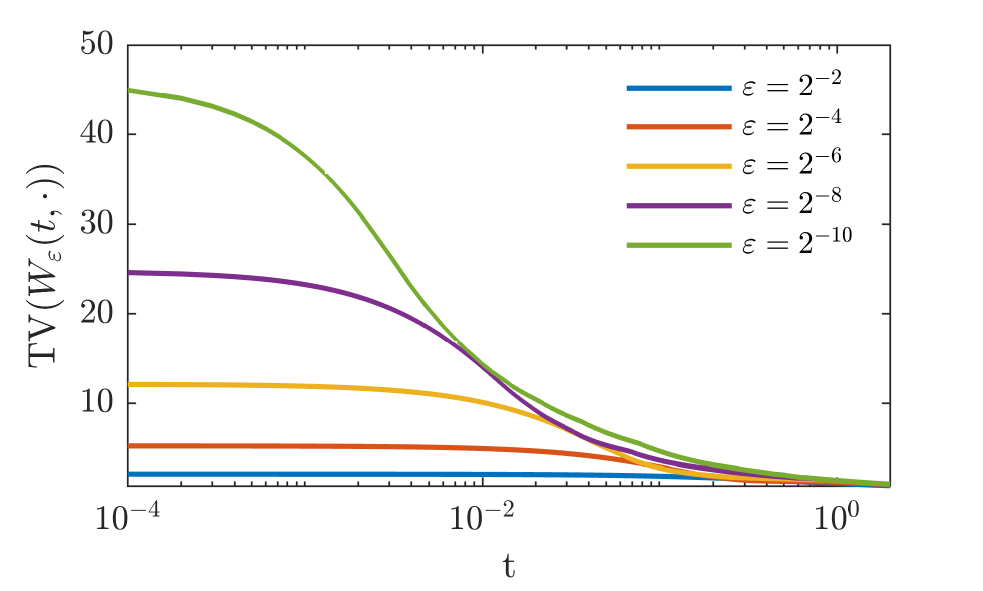}
    \caption{Illustration of $\mathrm{TV}(W_\eps(t,\cdot))$. Total variations of the nonlocal term $W_\eps$ for an initial datum with unbounded total variation, i.e.,  $\rho_0 \:= \sum_{n=1}^\infty \mathds{1}_{(1/{n+1},1/n+1+1/(2n(n+1)))}$, linear velocity $V(\xi) = 1-\xi$, and exponential kernel, i.e.,  $\eta(\cdot) := \eps^{-1}\exp(-\cdot \eps^{-1})$.}
    \label{fig:sim3}
\end{figure}

\section{Open problems}
\label{sec:open}

In this contribution, we proved several Ole\u{\i}nik-type inequalities for nonlocal conservation laws with exponential kernel.
 As a byproduct, we obtained some convergence results for the nonlocal-to-local limit problem without monotonicity or total variation assumptions on the initial data. Several questions remain open for future work: 
\begin{enumerate}
    \item the case of more general velocity functions (which, in turn, means more general initial data) that do not satisfy the technical assumptions in Theorems \ref{th:o} or \ref{th:og};
    \item the case of more general nonlocal weights (i.e., not necessarily of exponential type), as considered in \cite{2206.03949}.
\end{enumerate}

\vspace{5mm}
\section*{Acknowledgments}

We thank D.~Serre and E.~Zuazua for helpful remarks on the topics of this work. 

G.~M.~Coclite, N.~De~Nitti, E.~ Marconi and L.~V.~Spinolo are members of the Gruppo Nazionale per l'Analisi Matematica, la Probabilit\`a e le loro Applicazioni (GNAMPA) of the Istituto Nazionale di Alta Matematica (INdAM). G.~M.~Coclite has been partially supported by the  Research Project of National Relevance ``Multiscale Innovative Materials and Structures'' granted by the Italian Ministry of Education, University and Research (MIUR Prin 2017, project code 2017J4EAYB) and by the Italian Ministry of Education, University and Research under the Programme Department of Excellence Legge 232/2016 (Grant No. CUP - D94I18000260001). M.~Colombo is supported by the SNF Grant 182565. G.~Crippa is supported by the ERC StG 676675 FLIRT and by the SNF Project 212573 FLUTURA. N.~De Nitti has been partially supported by the Alexander von Humboldt Foundation and by the TRR-154 project of the Deutsche Forschungsgemeinschaft (DFG, German Research Foundation).  L.~Pflug has been supported by the DFG -- Project-ID 416229255 -- SFB 1411. E.~Marconi is supported by the Marie Sk\l odowska-Curie grant No. 101025032. L.~V.~Spinolo is a member of the PRIN 2020 Project 20204NT8W4.

\vspace{5mm}

\bibliographystyle{abbrv} 
\bibliography{NonlocalToLocal-ref.bib}
\vfill 

\end{document}